\documentclass{ws-ijac}

\newcommand{\be}{\begin{enumerate}}
\newcommand{\ee}{\end{enumerate}}
\newcommand{\beq}{\begin{equation}}
\newcommand{\eeq}{\end{equation}}

\begin{document}

\markboth{O. Kharlampovich, A. Myasnikov} {Definable sets in a  hyperbolic group}

%
\catchline{}{}{}{}{}
%

\title{Definable sets in a  hyperbolic group}

\author{Olga Kharlampovich}
\address{Department of Mathematics and Statistics, Hunter College, City University NY \\ New York 110065, USA\\ \email{okharlampovich@gmail.com}}

\author{Alexei Myasnikov}
\address{Department of Mathematics, Stevens Institute of Technology \\ Hoboken, NJ 07030, USA\\ \email{amiasnikov@gmail.com}}

\maketitle

\begin{history}
\received{(Day Month Year)} \revised{(Day Month Year)} \comby{[editor]}
\end{history}

\begin{abstract}
We give a description of definable sets $P=(p_1,\ldots ,p_m)$ in a free non-abelian group $F$ and in a torsion-free non-elementary hyperbolic group $G$. As a corollary we show that proper non-cyclic subgroups of $F$ and $G$ are not definable.
This answers Malcev's question posed in 1965 for $F$.
\end{abstract}

\section{Introduction}
We denote by $F$  a free group with finite basis and by $G$  a torsion-free non-elementary hyperbolic group
 and consider formulas in the  language $L_A$ that contains generators of $F$  (or $G$) as constants. In this paper we give a description of subsets of $F^m$  definable in  $F$ (Theorem \ref{general}) that follows from \cite{KMel} and similar description for $G$ (Theorem \ref{generalG}) that uses \cite{Sela7}.
A subset $S\in H^n$ is definable in a group H if there exists a
first-order formula $\phi(P)$ in $L_H$ such that $S$ is precisely the set of
tuples in $H^n$ where $\phi(P)$ holds:
$S = \{g\in H^n|H\vDash\phi (g)\}$.

 Our description implies that definable subsets in $F$  are either negligible or co-negligible (Bestvina and Feighn's result) and they are also either negligible or generic in the meaning of the complexity theory.  We will obtain the following corollary.

\begin{theorem}\label{co1} Proper non-cyclic subgroups of $F$ and $G$ are not definable.
\end{theorem}

These results solve Malcev's problem 1.19 from \cite{Kour} posed in 1965. Malcev asked the following:

1) Describe definable sets in $F$;

2) Describe definable subgroups in $F$;

3) Is the derived subgroup $[F,F]$ of F definable in $F$?

The main result, Theorem \ref{general}, will be proved in Section 4, Theorem \ref{co1} for $F$  will be proved in Section \ref{NS}.
In Section \ref{NS} we will also prove
\begin{theorem}\label{co2} The set of primitive elements of $F$ is not definable if $rank (F)>2.$
\end{theorem}

Notice that the set of all bases in $F_2 = F_2(a, b)$ is definable.
This is based on Nielsen's theorem: elements $g; h\in F_2$ form a
basis if and only if the commutator $[g, h]$ is conjugated either to $[a, b]$ or $[b, a]$.
Hence the set of bases in $F_2$ is defined by the following formula
$$\phi (p_1, p_2) = \exists z([p_1, p_2] = z^{-1}[a, b]z\vee [p_1, p_2] = z^{-1}[b, a]z).$$
The set of all primitive elements in
$F_2 = F_2(a, b)$ is defined by the following formula
$$\psi (p_1) = \exists p_2 \phi (p_1, p_2).$$

The results on hyperbolic groups will be presented in the last section.

\section{Conjunctive $\exists\forall$-formulas}

 The first step to analyze the structure of definable sets is to reduce it to the study of the structure of $\forall\exists$-definable sets.

\begin{theorem} \cite{Sela},\cite{KMel} Every formula in the theory of $F$ is equivalent to a boolean combination of $\forall\exists$-formulas.
\end{theorem}
Furthermore, a more precise result holds.
\begin{theorem} \label{prec}
 Every set definable in $F$ is defined by some boolean combination of
formulas \begin{equation}\label{AE}
\exists X\forall Y (U(P,X)=1\wedge \neg (V(P,X,Y)=1)),
\end{equation}where $X,Y,P$ are tuples of variables.
\end{theorem}
We call  formulas in the form (\ref{AE}) where $U(P,X)=1$ and $V(P,X,Y)=1$ are either equations or finite systems of equations, {\em conjunctive $\exists\forall$-formulas}.  Notice that in the language $L_A$  every finite system of equations in $F$ and in $G$ is equivalent to one equation. For $F$ this is Malcev's result, see also Lemma 3 in \cite{Imp}, for $G$ this is  Lemma \ref{onehyp}.   Therefore  we can assume that
$U(P,X)=1$ and $V(P,X,Y)=1$ are equations (although this is not essential for the proof of the main results). Every finite disjunction of equations in $F$ and $G$ is equivalent to one equation. This is attributed to Gurevich for $F$, see also Lemma 4 in \cite{Imp}, and the same proof works for $G$.

\begin{proof}  By Theorem 1, every definable set is defined by a boolean combination of $\exists\forall$-formulas. By Lemma 10 \cite{Imp}, every  $\exists\forall$-formula is equivalent to
$$ \exists X\forall Y (V(X,Y,P)=1\rightarrow U(Y)=1),$$
where $X,Y,P$ are tuples of variables.

This formula has form (25) in \cite{KMel}, namely
\begin{equation}\label{38} \phi (P)=\forall Z \exists X\forall Y (V(X,Y,P)=1\rightarrow U(Y)=1).\end{equation}
The proof of Theorem 39  in Section 12.7 of \cite{KMel} shows that the formula  $\phi (P)$ is false for a value $\bar P$ of the variables $P$ if and only if the conjunction of disjunctions of formulas of the  two types given below is true for $\bar P$. We will write these formulas in the same form as they appear in Section 12.7 of \cite{KMel}. Notice that instead of the union of variables $X_1,Y_1,\ldots ,X_{k-1}$ in these formulas we take variables $P$.

\begin{equation}\label{39}
\exists Z_{k-1}\forall B,C (U(P,Z_{k-1})=1\wedge V(P,Z_{k-1},B,C)\neq 1),
\end{equation}
\begin{equation}\label{40}
\forall Z_{k-1}\exists B(U'(P, Z_{k-1})=1\rightarrow V'(P,Z_{k-1},B)=1.\end{equation}

 The first formula is in the form (\ref{AE}). The negation of
the second formula is also in the form (\ref{AE}).

We give more details on how the formulas (\ref{39}) and (\ref{40}) are obtained in \cite{KMel} in the Appendix.
\end{proof}

\section{Diophantine sets and $\exists$-sets}

\begin{definition} A {\em piece} of a word $u\in F$ is a non-trivial subword $v$ that appears in at least two different ways (maybe the second time as $v^{-1}$, maybe with overlapping).

 A {\em piece} of a tuple of reduced words $(u_1,\ldots ,u_m)$, $u_j\in F$ is a non-trivial subword $v$ that appears in at least two different ways as a subword of  some words of $u_1,\ldots ,u_m$.
\end{definition}

\begin{definition}\label{def2}
A proper  subset $P$ of $F$ admits parametrization if it is a set of all  words $p$ that satisfy a given system of equations  (with coefficients) without cancellations in the form \begin{equation}\label{p}
p\circeq w_t(y_1,\ldots ,y_n),  t=1,\ldots ,k, \end{equation}
where for all $i=1,\ldots ,n$,  $y_i\neq 1$, each $y_i$ appears at least twice in the system and each variable $y_i$ in $w_1$ is a piece of $p$.

A proper  subset $P$ of $F^m$ admits parametrization if after permutation of indices it is a  product set of $F^k, k<m$ and a set of all  tuples of words $p_j, j=1,\ldots ,m-k$ that satisfy a given system of equations  (with coefficients) without cancellations in the form \begin{equation}\label{p}
p_j\circeq w_{tj}(y_1,\ldots ,y_n),  t=1,\ldots ,k_j, \end{equation}
where for all $i=1,\ldots ,n$,  $y_i\neq 1$, each $y_i$ appears at least twice in the system and each variable $y_i$ in each $w_{1j}$ is a piece of the tuple.

The empty set and one-element subsets of $F^m$ admit parametrization.

A finite union of sets admitting parametrization will be called a {\em  multipattern}. A subset of a multipattern will be called {\em a sub-multipattern}
\end{definition}

In this section we will give a description of Diophantine sets and $\exists$-sets in $F$.
\begin{theorem} \label{E-set} Suppose a  Diophantine set $P\subseteq F^m$  defined by the formula $$\psi (P)=\exists Y U(Y,P)=1,$$
where $U(Y,P)=1$ is a system of equations,
is not the whole group $F^m$. Then  $P$  is a multipattern.\end{theorem}

We will prove this result using the notion of a cut equation introduced in \cite{Imp}, Section 5.7.

\begin{definition}\label{df:cut}
A cut equation $\Pi = ({\mathcal E}, M, X, f_M, f_X)$ consists of
a set of intervals $\mathcal E$, a set of variables $M$, a set of
parameters $X$, and two  labeling functions $$f_X: {\mathcal E}
\rightarrow F[X], \ \ \ f_M:  {\mathcal E} \rightarrow F[M] .$$
For an interval $\sigma \in {\mathcal E}$ the image  $f_M(\sigma)
= f_M(\sigma)(M)$ is a reduced word in variables $M^{\pm 1}$ and
constants from $F$, we call it a {\it partition}  of
$f_X(\sigma)$.
\end{definition}

Sometimes we write $\Pi = ({\mathcal E}, f_M, f_X)$ omitting $M$
and $X$.

\begin{definition}
A solution of a cut equation $\Pi = ({\mathcal E}, f_M, f_X)$ with
respect to an $F$-homomorphism $\beta : F[X] \rightarrow F$  is an
$F$-homomorphism $\alpha : F[M] \rightarrow F$ such that: 1) for
every $\mu \in M$ $\alpha(\mu)$ is a reduced non-empty word; 2)
for every reduced word $f_M(\sigma)(M)\ (\sigma \in {\mathcal E})$
the replacement $m \rightarrow \alpha(m) \ (m \in M) $ results in
a word $f_M(\sigma)(\alpha(M))$ which is  a  reduced word as
written and  such that $f_M(\sigma)(\alpha(M))$ is graphically
equal to the reduced form of $\beta(f_X(\sigma))$; in particular,
the following diagram is commutative.
\begin{center}
\begin{picture}(100,100)(0,0)
\put(50,100){$\mathcal E$} \put(0,50){$F(X)$} \put(100,50){$F(M)$}
\put(50,0){$F$} \put(47,97){\vector(-1,-1){30}}
\put(53,97){\vector(1,-1){30}} \put(3,47){\vector(1,-1){30}}
\put(97,47){\vector(-1,-1){30}} \put(16,73){$f_X$}
\put(77,73){$f_M$} \put(16,23){$\beta$} \put(77,23){$\alpha$}
\end{picture}
\end{center}
\end{definition}

If $\alpha: F[M] \rightarrow F$ is a solution of a cut equation
$\Pi = ({\mathcal E}, f_M, f_X)$ with respect to an
$F$-homomorphism $\beta : F[X] \rightarrow F$, then we write
$(\Pi, \beta,\alpha)$ and refer to $\alpha$ as a \emph{solution
of} $\Pi$ {\it modulo} $\beta$. In this event,   for a given
$\sigma \in {\mathcal E}$ we say that $f_M(\sigma)(\alpha(M))$ is
a {\it partition} of $\beta(f_X(\sigma))$. Sometimes we also
consider homomorphisms $\alpha:F[M] \rightarrow F$, for which the
diagram above is still commutative, but cancellation may occur in
the words $f_M(\sigma)(\alpha(M))$. In this event we refer to
$\alpha$ as a {\em group} solution of $\Pi$ with respect to
$\beta$.

Definition of a generalized equation can be found in \cite{Imp} (Definition 8, \cite{Imp}). This is one of the principal objects in our work on equations in groups. The following result states that every generalized equation is equivalent to a certain cut equation.
\begin{lemma} [Lemma 34,\cite{Imp}]
\label{le:cut}
 For a generalized equation $\Omega(H)$  one can effectively construct a cut
equation
 $\Pi_{\Omega} = ({\mathcal E}, f_X, f_M)$ such that  the following
conditions hold:
\begin{enumerate}
 \item [(1)] $ X$ is a partition of the whole interval $[1,\rho_{\Omega}]$ into
disjoint
  closed subintervals;

 \item  [(2)] $M$ contains the set of variables $H$;

 \item [(3)] for any solution $U = (u_1, \ldots, u_\rho)$ of $\Omega$  the cut
equation
 $\Pi_{\Omega}$ has a solution
 $\alpha$ modulo the canonical homomorphism $\beta_U: F(X) \rightarrow
F$
 ($\beta_U(x) = u_i u_{i+1} \ldots u_j$ where $i,j$ are,
correspondingly,
  the left and the right  end-points of the interval $x$);

 \item [(4)] for any solution $(\beta,\alpha)$ of the cut equation $\Pi_{\Omega}$
the
 restriction of $\alpha$ on $H$ gives a solution of the generalized
equation
 $\Omega$.
 \end{enumerate}
 \end{lemma}

The proof given below is verbatim the one cited but it is given in the paper to make it more self-contained and because some features of the construction are implicitly used in the proof of Theorem \ref{E-set}. All undefined notions used in the proof can be found in \cite{Imp}.

\begin{proof}
 We begin with defining the sets $X$ and $M$. Recall that  a
closed interval of $\Omega$ is a union of closed sections of
$\Omega$.  Let $X$ be an arbitrary partition of the whole interval
$[1,\rho_{\Omega}]$ into closed subintervals (i.e., any two
intervals in $X$ are disjoint and the union of $X$ is the whole
interval $[1,\rho_{\Omega}]$).

 Let $B$ be a set of representatives of dual bases of $\Omega$, i.e.,
for every base
  $\mu$  of $\Omega$ either $\mu$ or $\Delta(\mu)$ belongs to $B$, but
not both.
 Put $M = H \cup B$.

Now let $\sigma \in X$.  We denote  by $B_{\sigma}$  the set of
all bases over $\sigma$ and by  $H_\sigma$ the set of all items in
$\sigma$. Put $S_\sigma = B_{\sigma} \cup H_\sigma.$  For $e \in
S_\sigma$ let $s(e)$ be the interval $[i,j]$, where $i < j$ are
the endpoints of $e$.  A sequence  $P = (e_1, \ldots,e_k)$ of
elements from $S_\sigma$  is called a {\it partition} of $\sigma$
if $s(e_1) \cup \cdots \cup s(e_k) = \sigma$ and $s(e_i) \cap
s(e_j) = \emptyset$ for $i \neq j$.  Let ${\rm Part}_\sigma$ be
the set of all partitions of $\sigma$. Now put
 $${\mathcal E} = \{P \mid P \in {\rm Part}_\sigma, \sigma \in X\}.$$
Then for every $P \in {\mathcal E}$
 there exists one and only one $\sigma \in X$ such that $P \in
{\rm Part}_\sigma$.
 Denote this $\sigma$ by $f_X(P)$. The map $f_X: P \rightarrow f_X(P)$
is a
 well-defined function from ${\mathcal E}$ into $F(X)$.

 Each partition $P = (e_1, \ldots,e_k) \in {\rm Part}_\sigma$ gives rise to a
word
  $w_P(M) = w_1 \ldots w_k$  as follows.  If $e_i \in H_\sigma$ then
$w_i = e_i$.
  If $e_i = \mu \in B_\sigma$ then $w_i = \mu^{\varepsilon(\mu)}$.
  If $e_i = \mu$ and $\Delta(\mu) \in B_\sigma$ then
  $w_i = \Delta(\mu)^{\varepsilon(\mu)}$. The map $f_M(P) = w_P(M)$ is a
  well-defined function from ${\mathcal E}$ into $F(M)$.

  Now set $\Pi_{\Omega} = ({\mathcal E}, f_X, f_M)$.  It is not hard to see
from the
  construction that the cut equation $\Pi_\Omega$ satisfies all the
required properties.
  Indeed, (1) and (2) follow directly from the construction.

  To verify (3), let's consider a
  solution $U = (u_1, \ldots, u_{\rho_\Omega})$ of $\Omega$. To define
  corresponding functions $\beta_U$ and $\alpha$, observe that the
function $s(e)$
  (see above) is defined for every $e \in X \cup M$.
   Now  for $\sigma \in X$ put
 $\beta_U(\sigma) = u_i\ldots u_j$, where $s(\sigma) = [i,j]$, and for $m
\in M$
 put $\alpha(m) =  u_i\ldots u_j$, where $s(m) = [i,j]$.  Clearly,
$\alpha$ is a
 solution of $\Pi_\Omega$ modulo $\beta$.

 To verify (4) observe that if $\alpha$ is a solution of $\Pi_\Omega$ modulo
$\beta$,
  then the restriction of $\alpha$ onto the subset $H \subset M$ gives a
solution
 of the generalized equation $\Omega$. This follows from the
construction
 of the words $w_p$ and the fact that the words $w_p(\alpha(M))$ are
reduced as
 written (see definition of a solution of a cut equation). Indeed, if a
base
 $\mu$ occurs in a partition $P \in {\mathcal E}$, then there is a partition
 $P^\prime \in {\mathcal E}$  which is obtained from $P$ by replacing $\mu$
by the
 sequence $h_i\ldots h_j$. Since there is no cancellation in words
$w_P(\alpha(M))$ and
 $w_{P^\prime}(\alpha(M))$, this implies that
 $\alpha(\mu)^{\varepsilon(\mu)} = \alpha(h_i\ldots h_j)$. This shows
that
 $\alpha_H$ is a solution of $\Omega$.
 \end{proof}

We will give an example. Let $\Omega$ be a generalized equation in Fig \ref{fig1}.
\begin{figure}[h]
\centering{\mbox{\psfig{figure=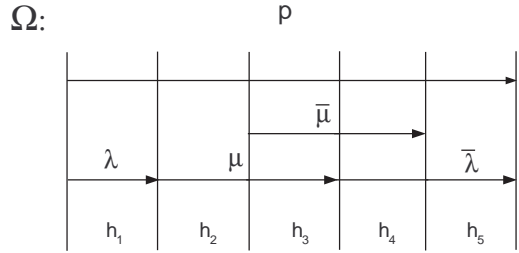,height=1.6in}}}
\caption{Generalized equation $\Omega$}
\label{fig1}
\end{figure}
 For the cut equation we will have $X=\{p\}$, $M=\{h_1,h_2,h_3,h_4,h_5, \lambda, \mu\},$ ${\mathcal E} =\{\sigma _1,
\sigma _2,\sigma _3,\sigma _4,\sigma _5\}$, $f_{X}(\sigma _i)=p,\  i=1,\ldots ,5.$ Further, $f_M(\sigma _1)=\lambda\mu\lambda,$
 $f_M(\sigma _2)=h_1h_2\mu h_5,$ $f_M(\sigma _3)=h_1h_2h_3h_4h_5,$ $f_M(\sigma _4)=h_1\mu\lambda,$ $f_M(\sigma _5)=h_1h_2h_3\lambda.$
 This gives five partitions for $p$.

\begin{lemma} [Theorem 8, \cite{Imp}] \label{th:cut}
Let $S(X,Y,A)) = 1$ be a system of  equations over $F = F(A)$.
Then one can effectively construct a finite set of cut equations
 $${\mathcal CE}(S) = \{\Pi_i \mid  \Pi_i =({\mathcal E}_i, f_{X_i},  f_{M_i}),
i = 1 \ldots,k \}$$
  and a finite set of tuples of words $\{Q_i(M_i) \mid i = 1,
\dots,k\}$  such that:
\begin{enumerate}
\item for every equation $\Pi_i =({\mathcal E}_i, f_{X_i},
f_{M_i}) \in {\mathcal CE}(S)$,
 one has $X_i = X$ and   $f_{X_i}({\mathcal E}_i) \subset  X^{\pm 1}$;

\item  for  any  solution $(U,V)$ of $S(X,Y,A) = 1$ in $F(A)$,
there exists a number $i$
 and a tuple of words $P_{i,V}$ such that the  cut
equation $\Pi_i \in {\mathcal CE}(S)$  has a solution $\alpha: M_i
\rightarrow F$ with respect to the $F$-homomorphism  $\beta_U:
F[X] \rightarrow F$ which is induced by the map $ X \rightarrow
U$. Moreover, $U = Q_i(\alpha(M_i))$, the word $Q_i(\alpha(M_i))$
is reduced as written,  and  $V = P_{i,V}(\alpha(M_i))$;

\item  for any $\Pi_i \in {\mathcal CE}(S)$ there exists a tuple
of words $P_{i,V}$ such that for  any solution (group solution)
$(\beta, \alpha)$ of  $\Pi_i$ the pair $(U,V),$ where $U =
Q_i(\alpha(M_i))$ and $V = P_{i,V}(\alpha(M_i)),$ is a solution of
$S(X,Y) = 1$ in $F$.

\end{enumerate}
\end{lemma}

{\bf Proof of Theorem \ref{E-set}.}
We think about  $U(Y,P)=1$ as a system of equations with coefficients $A$ and two sets of variables: $X=P$ and $Y$.
It follows from Lemma \ref{th:cut} (where the role of $S(X,Y,A)=1$ is taken by the system $U(P,Y)=1$), that one can effectively  construct a finite set of cut equations satisfying the conditions of the statement.  Consider one of the cut equations. We can assume that this cut equation has a solution.
It follows from the construction that there is at least one  interval $\sigma _j\in{\mathcal E}$ labeled by $p_j$ that is completely covered by variables from $M$ and constants (=coefficients).  We can assume that  each variable $y_i$ from $M$ appears at least twice in  the cut equation, otherwise we can remove the partition containing $y_i$ and express $y_i$ in terms of other variables and $p_i$. After removing the partition solutions satisfying this cut equation will not be  graphical solutions of the cut equation, but they will be  group solutions, and every group solution of the cut equation still provides a solution of the initial equation S(X,Y,A)=1 (see item 3 in Theorem \ref{th:cut}).
(Let us consider the example given above. We have $$p=\lambda\mu\lambda = h_1h_2\mu h_5=h_1h_2h_3h_4h_5=h_1\mu\lambda =h_1h_2h_3\lambda.$$
We remove $p=h_1h_2h_3h_4h_5,$ because $h_4$ is contained only once, then remove $p=h_1h_2\mu h_5$, because now $h_5$ is contained only once, then remove $p=h_1h_2h_3\lambda$, then $p=h_1\mu\lambda,$
and, finally, $p=\lambda\mu\lambda.$ Therefore $p$ can be arbitrary element of $F$.)

By assumption, formula $\psi (P)$ does not define the whole group   $F^{m}$. If formula $\psi (P)$  defines a finite or empty set, then the set $P$ is trivially a multipattern. Hence, we further assume that $\psi (P)$ does not define a finite or empty set. So at least one $f_X({\sigma _j})$  can be represented in several ways as a product without cancellation in the form (\ref{p}). Variables from $M$ correspond to pieces of a reduced word $f_X({\sigma _j})$ corresponding to $p_j$.   Suppose  there exists $p=p_j$ such that for each partition  $p_j=w_{ij},$ where $i=1,\ldots , k_j$ in (\ref{p}) there is a variable that appears only as a matching variable. Let $y$ be a variable in the partition $p_j=w_{1j}$ that appears only as a matching variable. Cutting $y$, if necessary, we can assume it does not intersect any boundaries. If it is covered by some variable $z$ in some other partition of $p_j$, and $z$ has a non-matching occurrence, then $y$ is a piece (see Figure \ref{fig2}).
\begin{figure}[h]
\centering{\mbox{\psfig{figure=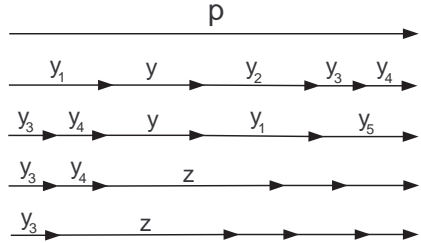,height=1.6in}}}
\caption{Matching variable $y$ is a piece}
\label{fig2}
\end{figure}

If all variables $z$ that cover $y$ in other partitions of $p_j$ only have matching occurrences, we represent each such $z$ as $z=z_1yz_2$.  To find a group solution of the cut equation we can now remove all occurrences of $y$,  solve the remaining system of equations for the remaining variables, and then express the value of the removed variable $y$ in terms of the remaining variables and $p_j$ (see Figure \ref{fig3}).
  \begin{figure}[h]
\centering{\mbox{\psfig{figure=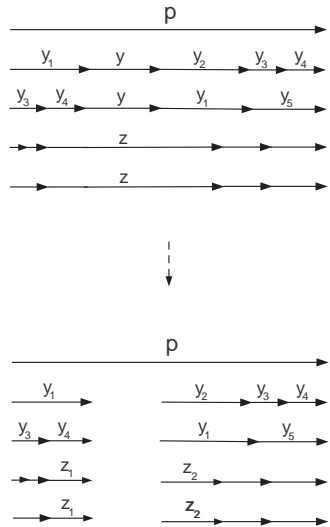,height=3in}}}
\caption{Matching variable $y$ is removed}
\label{fig3}
\end{figure}
  We can take any $p_j\in F$ in such a group solution provided the remaining variables satisfy the system of equations obtained by removing the matching variable $y$ (We denote this system by $Q(M)$).  This system has a solution, because we assumed that the cut equation has a solution. Therefore,  $p_j$ can be arbitrary element of $F$. If $m=1$ and $F^m=F$ this  contradicts to the assumption of the theorem, therefore in this case at least in one partition of $p_1=p_j$, $p_1=w_1(y_1,\ldots ,y_n)$, each variable is a piece. This proves for $m=1$ that the words $w_{t1}$ give a parametrization of the set $P$ satisfying properties of Definition \ref{def2} and hence $P$ is a multipattern. Therefore this proves  the theorem for $m=1$.  If $m>1$, we can assume that $j=m$, and $p_m$ can be arbitrary element of $F$. The existence of a group solution of the cut equation we began with for $p_1,\ldots ,p_{m-1}$ is equivalent to the existence of a solution of the system that consists of (\ref{p}) for $j=1,\ldots ,m-1$ and $Q(M)=1$.  We can now use induction on $m$. Since $\psi (P)$ defines neither $F^m$ nor a finite set nor an empty set, it follows by induction that
for  $p_j, j=1,\ldots , m-k$ words $w_{tj}$ give a parametrization of the set $P$ satisfying properties of Definition \ref{def2} and hence $P$ is a multipattern. Theorem \ref{E-set} is proved.

\begin{theorem}\label{EE-set}
  Suppose an  $\exists$-set $P$ is defined by the formula $$\psi _1(P)=\exists Y (U(Y,P)=1\wedge V(Y,P)\neq 1).$$  If the positive formula $\psi (P)=\exists Y (U(Y,P)=1$ does not define the   whole group $F^m$, then $P$ is a sub-multipattern. Otherwise, either $\neg P$ or $P$ is a sub-multipattern.
\end{theorem}

\begin{proof} If the positive formula $\psi (P)=\exists Y (U(Y,P)=1$ does not define the   whole group $F^m$, then by Theorem \ref{E-set} it defines a multipattern, and $P$ is a subset of this multipattern, therefore, a sub-multipattern.

Suppose now that $\psi (P)$ defines  the whole group $F^m$. Then the set of parameters defined by $\psi _1(P)$ contains the set of parameters defined by $\psi _2(P)=\exists Y_1 U_1(Y_1)=1\wedge V_1(Y_1,P)\neq 1.$ Let us prove this.
Parameters satisfying $\psi (P)$ are described by disjunction of cut equations, and one of the cut equations defines the whole group $F^m$ (denote it $\Pi _1$), therefore this cut equation can be described by the formula $\exists M U_1(M)=1$. This  follows from \cite{Imp}, Lemma 8 and also can be seen directly. Indeed, if $y$ has only matching occurrences we remove it from all partitions of  $p_j$, and  we have a  system of equations for the remaining variables that does not contain $p_j$.
Every tuple from $Y$ corresponding to the cut equation $\Pi _1$  can be represented as a function $Y=f(M,P,A).$ The system of equations consisting of $V(f(M,P,A),P)=1$ for all such functions $f(M,P,A)$, by the Noetherian property, is equivalent to one equation $V_1(M,P)=1$. If $\psi _1(\bar P)$ is false, then $\bar P$ satisfies the formula.
$$\neg\psi _2(\bar P)=\forall M (U_1(M)=1\rightarrow V_1(M,\bar P)=1).$$  Since there exists $M$ such that $U_1(M)=1,$ such tuples $\bar P$ either constitute the whole $F^m$ or form a disjunction of cut equations, therefore form a multipattern.
If $\neg\psi_2(P)$ defines a multipattern, then $\neg\psi_1(P)$ defines a sub-multipattern.
If for any $\bar P\in F^m$ we have $\neg\psi_2(\bar P)$, then $\psi_2(P)$ defines an empty set and we cannot use solutions of the cut equation $\Pi _1$ to obtain solutions of $U(Y,P)=1\wedge V(P,Y)\neq 1$. Then we consider the next cut equation $\Pi _2$ corresponding to the equation $U(Y,P)=1$ and defining the whole group $F^m$ (if exists) and  construct $\psi _2(P)$ the same way as we constructed it for $\Pi _1$. If one of the formulas  $\neg\psi_2(P)$ defines a multipattern, then $\neg P$ is a sub-multipattern, otherwise $P$ is a sub-multipattern.
\end{proof}

\section{Main Theorem}

In this section we will prove the main theorem.

\begin{theorem} \label{general} For every definable set $P\subseteq F^m$ in a free group $F$, either $P$ or its complement $\neg P$ is a sub-multipattern.
\end{theorem}

\begin{proof}
 Suppose a set $P$ is defined by the formula (\ref{AE}).
If the
$\exists$-set defined by $\exists X(U(P,X)=1$) is not the whole group $F^m$, then the set $P$ defined by the  formula (\ref{AE}) is a sub-multipattern.

Suppose now that the set defined by $\exists X(U(P,X)=1$ is the whole group, then, as in the proof of Theorem \ref{EE-set}, a subset of parameters satisfying formula (\ref{AE}) is a union of a sub-multipattern and another subset that is defined by
$$\exists X_1\forall Y (U_1(X_1)=1\wedge V_1(X_1,Y,P)\neq 1).$$ Suppose this formula does not define the empty set. Then the negation  is
$$\phi _1(P)=\forall X_1\exists Y (U_1(X_1)=1\rightarrow V_1(X_1,Y,P)=1)$$
and it does not define $F^m$.

\begin{lemma} Formula $$\theta (P)=\forall X_1\exists Y (U_1(X_1)=1\rightarrow V_1(X_1,Y,P)=1)$$  in $F$ in the language $L_A$ is equivalent to the condition that  $\exists$-formula $\exists Y V_1(X_1,Y,P)=1$ holds in each of the finite number of NTQ groups $N_1,\ldots ,N_k$.
\end{lemma}
\begin{proof}
We can assume that the equation $U_1(X_1)=1$ is irreducible. By \cite{Imp} each solution of this equation factors through one of the finite number of $NTQ$-systems. By \cite{Imp}, Theorem 12, formula $\exists Y V_1(P,Y)=1$ holds in one of the corrective normalizing extensions $N_1,\ldots ,N_k$ of one of these NTQ systems. For any value $\bar P$ of variables $P$ that makes each such formula true in each of the NTQ groups $N_1,\ldots ,N_k$, this value also makes $\theta (P)$ true.

\end{proof}
Since $\neg P\neq F^m$, this lemma implies that a formula $\exists Y V_2(P,Y)=1$ holds in $F$, therefore by Theorem \ref{E-set} $\neg P$ must be a multi-pattern.

\end{proof}

\section{Negligible sets}\label{NS}

\begin{definition}\cite{BF}\label{1}
A subset $P$ of $F$ is {\em negligible} if there exists $\epsilon >0$ such that all but finitely many $p\in P$ have a piece such that $$\frac {length (piece)}{length (p)}\geq\epsilon.$$

A complement of a negligible subset is {\em co-negligible}.
\end{definition}

Bestvina and Feighn \cite{BF} stated that in the language without constants every definable subset of $F$ is either negligible or co-negligible.
They also proved
\begin{proposition} \cite{BF}\label{BF3} 1) Subsets of negligible sets are negligible.

2) Finite sets are negligible.

3) A subset $S$ containing a coset of a non-cyclic subgroup $G$ of $F$ cannot be negligible

4) A proper non-cyclic subgroup of $F$ is neither negligible nor co-negligible.

5) The set of primitive elements of $F$ is neither negligible nor co-negligible if $rank (F)>2.$
\end{proposition}
\begin{proof} Statements 1) and 2) immediately follow from the definition. 3) If $x,y\in G$ and $[x,y]\neq 1$, then
the infinite set $\{fxyxy^{2}x\ldots xy^{i}x, \ i\in{\mathbb N}\}$ is not negligible . Statement 4) follows from 3). 5) Let $a,b,c$ be three elements in the basis of $F$ and denote $F_2=F(a,b)$  The set of primitive elements contains $cF_2,$ and the complement contains $<[a,b], c^{-1}[a,b]c>.$

\end{proof}

\begin{lemma}
A set $P\subseteq F$ that is a sub-multipattern, is negligible.
\end{lemma}
\begin{proof}  It is enough to show that a set $P\subseteq F$ that admits parametrization is negligible. Let $m$ be the length of word $w_1$ (as a word in variables $y_i$'s and constants). The set $P$  is negligible with $\epsilon =1/m$.
\end{proof}

\begin{corollary} Every definable subset of $F$ in the language with constants (and, therefore, in the language without constants) is either negligible or co-negligible.\end{corollary}

This and Proposition \ref{BF3} imply Theorem\ref{co1} for $F$ and Theorem \ref{co2}.

\begin{definition}\label{2}
Recall that in complexity theory $T\subseteq F(X)$ is called generic if
$$\rho _n(T)=\frac{|T\cap B_n(X)|}{|B_n|}\rightarrow 1, \ if \ n\rightarrow\infty ,$$
where  $ B_n(X)$ is the ball of radius $n$ in the Cayley graph of $F(X)$. The term ``negligible'' is usually used for a complement of a generic set.
We will call in this paper such a set CT-negligible.
\end{definition}
\begin{proposition} Negligible sets in Definition \ref{1} are CT-negligible.

\end{proposition}
\begin{proof}
Let $2|X|=k$.  Then  $|B_n|=(k/k-2)((k-1)^n-1).$  Fix $\epsilon >0.$ Let $N$ be the set of words $w$ that have a piece
 such that $$\frac {|piece|}{|w|}\geq\epsilon.$$ If $|w|=n$, then $|piece|\geq m=\epsilon n.$ We now count the number
  of reduced words of length $n$ that have a piece of length $m$ (if they have a piece longer than $m$, then they also have a piece of length $m$). There are $\frac {n(n-1)}{2}$ choices for positions of the first letters of the pieces. Suppose these positions  are fixed,  one piece begins at position $i$ and the other at position $j$. Then there are two cases :

 1) If the pieces of length $m$ do not overlap, then up to a constant there are at most $(k-1)^{n-2m}(k-1)^m=(k-1)^{n-m}$ such words;

 2) If the pieces of length $m$  overlap, and $j=i+t$, then up to a constant there are at most $(k-1)^{n-m-t}(k-1)^t=(k-1)^{n-m} $ such words.

 Therefore up to a constant there are at most $n(n-1)(k-1)^{n-m}$ reduced words of length $n$ that have a piece of length $m=\epsilon n$.
  Let $P_{n,\epsilon}$ be the set of reduced words in $B_n$ that have a piece of length $m$, $m=\epsilon n.$
 $|P_{n,\epsilon}|$ is  at most $C\sum _{i=1}^ni(i-1)(k-1)^{i(1-\epsilon)}$, where $C$ is a constant.

 It is known that $$\sum _{i=1}^n iz^i=z\frac {1-(n+1)z^n+nz^{n+1}}{(1-z)^2},$$

 $$\sum _{i=1}^n i^2z^i=z\frac {1+z-(n+1)^2z^n+(2n^2+2n-1)z^{n+1}-n^2z^{n+2}}{(1-z)^3}.$$

 Using these formulas with $z=(k-1)^{(1-\epsilon )}$, we obtain that  $|P_{n,\epsilon}|\leq C_1n^2(k-1)^{(1-\epsilon)n},$ where $C_1$ is a constant.

 Then  $\rho _n\leq \frac{ C_1n^2(k-1)^{(1-\epsilon)n}}{(k-1)^n}=C_1\frac{n^2}{(k-1)^{n\epsilon}}.$  Therefore $\rho _n\rightarrow 0$, as $n\rightarrow\infty.$
\end{proof}

\section{Torsion free hyperbolic groups}\label{hyp}

In this section $G$ is a non-elementary torsion free hyperbolic group.
\begin{theorem}  Let $G$ be a non-elementary torsion free hyperbolic group. Every set definable in $G$ is defined by some boolean combination of
 conjunctive $\exists\forall$-formulas.
\end{theorem}
\begin{proof} Similarly to Theorem 6.5 \cite{Sela7} one can prove that $EAE(p)$ (we use Sela's notation) is in the boolean algebra of conjunctive $\exists\forall$ sets. Indeed $EAE(p)=T_1(p)\cup\ldots T_d(p),$  where $d$ is the depth of the tree of stratified sets and $T_n(p)$ is the set of specializations of the defining parameters $P$ for which there exists a valid PS statement for some proof system of depth $n$.  Lemma 6.2 in \cite{Sela7} deals with $T_1(p)$. The proof that $T_1(p)$ is a conjunctive $\exists\forall$ set is identical to the free group case. Theorem 6.3 \cite{Sela7} deals with $T_2(p)$. As in the free group case, Proposition 3.7 \cite{Sela51} (the proof of this proposition is not given there but it is stated that it is identical to the proof of Proposition 1.34 \cite{Sela51}) reduces the analysis of the set $T_2(p)$  to the set of specializations of the defining parameters $P=<p>$ for which there exists a test sequence of valid $PS$ statements that factors through the various resolutions $PSHGHRes$. By construction, if $p_0\in T_2(p)$ then there must exist a valid PS statement of the form
$$(r,(h_1^{2},g_1^{1}),\ldots ,(h_{d(ps)}^{2},g_{d(ps)}^{1}),h_0^1,w_0,p_0,a)$$ that factors through one of the PS resolutions PSHGHRes constructed with respect to all proof systems of depth 2. (Notice that the notion of a resolution corresponds to the notion of a fundamental sequence in our work.)  The sets $TSPS(p)$ associated with various PS resolutions  $PSHGHRes$, i.e. the sets of specializations $p_0$ of defining parameters $P=<p>$ for which there exists a test sequence (test sequence corresponds to a generic family) that factors through any of the PS resolutions $PSHGHRes$ and restricts to valid PS statements are in the Boolean algebra of $\exists\forall$ sets (Proposition 1.34, \cite{Sela51}). Moreover, it follows from the proof of Proosition 1.34 that for any specialization $p_0$ of the defining parameters, there are finitely many combinations for the collections of ungraded resolutions covered by a PS resolution  $PSHGHRes$, and the collections of ungraded resolutions covered by the other (auxiliary) graded resolutions associated with $PSHGHRes$. These finitely many possibilities for the collections of ungraded resolutions form a stratification on the set of specializations of the defining parameters, obtained from the bases of all the graded resolutions that have been constructed. Therefore $p_0\in TSPS(p)$ if and only if it belongs to certain strata in the combined stratification, and not to the complement of these strata, but it depends only on the stratum, not on the particular specialization. A stratum in the stratification is the set of specializations for which there exists a given combination of rigid and strictly solid families of specializations (= Max-classes) of finitely many rigid or solid limit groups (= groups without sufficient splitting). These sets of specializations can be defined by a Boolean combination of conjunctive $\exists\forall$ formulas. By Theorem 1.33 in \cite{Sela51}, if there exists a valid PS statement that factors through  a PS resolution
$PSHGHRes$, then either there exists a test sequence that factors through this resolution and  restricts to valid PS statements, or there must exist a combined specialization that factors through a resolution of lower complexity, and we can continue with this resolution. The definition of a valid PS statement is given in Definition 1.23 \cite{Sela51}. The fact that for a specialization $p_0$ there exists a valid PS statement that factors through  a PS resolution
$PSHGHRes$ can be expressed by a Boolean combination of conjunctive $\exists\forall$ formulas, because conditions (i)-(iv) in this definition can be expressed by such formulas.

Each set $T_n(p)$ is in the boolean algebra of conjunctive $\exists\forall$ sets, by applying the same sieve procedure that is used to analyze the set $T_2(p)$.
\end{proof}

\begin{definition} Let $G$ be  a  torsion-free hyperbolic group generated by a set $A$ and
$\nu: F(A) \to G$ the canonical projection.

\begin{itemize}
\item [1)]  A proper  subset $P$ of $G$ admits parametrization if $P$ is the image under $\nu$ of a set $\tilde P$ in $F(A)$ that admits parametrization in $F(A)$ and there exist constants $\lambda, c $ and $D$ such that for each $p\in P$ there is a pre-image $\tilde p\in\tilde P$ such that the path corresponding to $\tilde p$ in the Cayley graph of $G$ is $(\lambda, c)$-quasigeodesic in $D$-neighborhood of the geodesic path for $p$.

\item [2)]  A finite union of sets admitting parametrization is called a multipattern. A subset of a multipattern is a  sub-multipattern.
\end{itemize}
\end{definition}

A similar definition can be given for sets of tuples of elements of $G$.

\begin{definition} Let $G$ be  a  torsion-free hyperbolic group generated by a set $A$ and
$\nu: F(A) \to G$ the canonical projection.

\begin{itemize}
\item [1)]  A proper  subset $P$ of $G^m$ admits parametrization if $P$ is the image of the set $\tilde P$ in $F^m$ that admits parametrization in $F$ and there exist constants $\lambda, c$ and $D$ such that for each $p=(p_1,\ldots ,p_m)\in P$ there is a pre-image $\tilde p=(\tilde p_1,\ldots ,\tilde p_m)\in\tilde P$ such that the path corresponding to each  $\tilde p_i$ in the Cayley graph of $G$ is $(\lambda , c)$-quasigeodesic in $D$ neighborhood of the geodesic $p_i$.

\item [2)]  A finite union of sets admitting parametrization is called a multipattern. A subset of a multipattern is a  sub-multipattern.
\end{itemize}

\end{definition}

\begin{theorem} \label{generalG}  For every definable subset $P$ of non-elementary torsion free hyperbolic group $G^m$, $P$ or its complement $\neg P$ is a sub-multipattern.
\end{theorem}
\begin{proof} We will first show the set defined by the positive existential formula
$$\exists Y U(P,Y,A)=1,$$ where elements in $A$ are constants,
is a  sub-multipattern if it is not the whole group.

In \cite{RS95}, the problem of deciding whether or not a system of equations $S$ over a torsion-free hyperbolic group $G$ has a solution was solved by constructing quasigeodesics called
\emph{canonical representatives} for certain elements of $G$. This construction reduced the problem to deciding the existence of solutions in finitely many
systems of equations over free groups.  The reduction may also be used to describe parameters $P$ as shown
below.

\begin{lemma}\label{Lem:RipsSela1}\cite{RS95}
Let $G=<A|R>$ be a torsion-free $\delta$-hyperbolic group and $\pi : F(A)\rightarrow G$ the canonical epimorphism.  There
is an algorithm that, given a system $S(Z,A)=1$
of equations over $\Gamma$, produces finitely many systems of
equations
\begin{equation}
S_{1} (X_{1},A)=1,\ldots,S_{n}(X_{n},A)=1
\end{equation}
over $F$ and homomorphisms $\rho_{i}: F(Z,A)\rightarrow F_{R(S_{i})}$ for $i=1,\ldots,n$
such that
\begin{enumerate}
\item for every $F$-homomorphism $\phi : F_{R(S_{i})}\rightarrow F$,  the induced map $\overline{\rho_{i}\phi\pi}:\Gamma_{R(S)}\rightarrow \Gamma$ is a $\Gamma$-homomorphism, and
\item for every $\Gamma$-homomorphism $\psi: \Gamma_{R(S)}\rightarrow \Gamma$ there is an integer $i$ and an $F$-homomorphism
$\phi : F_{R(S_{i})}\rightarrow F(A)$ such that $\overline{\rho_{i}\phi\pi}=\psi$.
\end{enumerate}

Moreover, the algorithm also gives positive numbers $\lambda, c$ and $D$ such that for each solution $\psi (Z)$ of $S(Z,A)=1$ the corresponding words
$\phi (\rho_{i} (Z))$ in $F(A)$ represent  $(\lambda , c)$-quasigeodesics in the $D$ neighborhood of corresponding elements in $\psi (Z)$.

\end{lemma}

By this lemma the set of parameters $P$ defined by the formula $\exists Y U(P,Y,A)=1$ consists of the images of the set of parameters satisfying the formula  $\exists Y W(Y,P,A)=1$ for certain system of equations $\exists Y W(Y,P,A)=1$ is $F(A).$ This proves that this set is a sub-multipattern.

Now the proof of the theorem is the same as in the free group case.  Suppose a set $P$ is defined by the formula (\ref{AE}).
If the
$\exists$-set defined by $\exists X(U(P,X)=1$) is not the whole group $G^m$, then the set $P$ defined by the  formula (\ref{AE}) is a sub-multipattern.

Suppose now that the set defined by $\exists X(U(P,X)=1$ is the whole group, then, as in the proof of Theorem \ref{general}, a subset of parameters satisfying formula (\ref{AE}) is a union of a sub-multipattern and another subset that is defined by
$$\exists X_1\forall Y (U_1(X_1)=1\wedge V_1(X_1,Y,P)\neq 1).$$ Suppose this formula does not define the empty set. Then the negation  is
$$\phi _1(P)=\forall X_1\exists Y (U_1(X_1)=1\rightarrow V_1(X_1,Y,P)=1)$$
and it does not define $G^m$.

\begin{lemma} (\cite{Sela7}, Theorem 2.3) Formula $$\theta (P)=\forall X_1\exists Y (U_1(X_1)=1\rightarrow V_1(X_1,Y,P)=1)$$  in $G$ in the language $L_A$ is equivalent to the condition that  there exists a formal solution $Y$ of the system $V_1(X_1,Y,P)=1$  in the covering closure (which corresponds to a finite number of NTQ groups $N_1,\ldots ,N_k$).
\end{lemma}

Since $\neg P\neq G^m$, this lemma implies that a formula $\exists Y V_2(P,Y)=1$ holds in $G$, therefore  $\neg P$ must be a multi-pattern.
\end{proof}
\begin{theorem} Proper non-cyclic subgroups in a non-elementary torsion free hyperbolic group $G$ are not definable.
\end{theorem}
\begin{proof} We will prove the theorem by contradiction.
Let $H$ be a definable non-cyclic subgroup in $G$. Let $a,b$ be two non-commuting elements in $H$ such that they generate a free subgroup, we can assume that $b$ is cyclically minimal. Such elements exist by \cite{Arzh}, Lemma 1.14. Let $x=a^n, y=b^m$ where $n,m$ are sufficiently large numbers so that the set of elements  $\{P_i=xyxy^{2}x\ldots xy^{i}x, \ i\in{\mathbb N}\}$ consists of $(\lambda _1, c_1)$-quasigeodesics in the $D_1$-neighborhood of corresponding geodesics. Such numbers $n,m$ exist by \cite{Arzh}, Lemma 2.3.
We can assume that $\lambda _1=\lambda, c=c_1$ and $D_1=D$.

Elements $\{P_{i}\}$ can be represented by  quasigeodesics $Q_{i}$ that have pieces $R_i=\bar R_i$ of length  greater than $\epsilon |Q_{i}|$.
 We have one of the following two cases.

 1) There is a number $\delta$ and a subsequence of indices $\{i_j\}$ such that the non-overlapping part of  $R_{i_j}$ and $\bar R_{i_j}$ in $Q_{i_j}$ is greater than $\delta |Q_{i}|$.

 2) There is a subsequence of indices $\{i_j\}$ such that  $Q_{i_j}$  have
periodic subwords of length greater than $\epsilon |Q_{i_j}|$, with period $q_{ij}$ such that $|q_{ij}|/|Q_{ij}|\rightarrow 0$.

In the second case, since quasigeodesics $P_{i_j}$ and $Q_{i_j}$  are $2D$ close to each other, such  long periodic subwords must appear in the elements $P_{i_j}$. Indeed, the number of different geodesics joining phase vertices of the subpath labeled by $q_{ij}^k$ of the path labeled by $Q_{ij}$ to the nearest vertices of the path labeled by $P_{ij}$ in the generators $a$ and $b$ is bounded by $|2A|^{2D+|a|+|b|}.$ Therefore, the same geodesic (labeled, say, by $t$) keeps repeating, and  a relatively large part of $P_{ij}$ (greater than $\epsilon |P_{ij}|$) is a product of commuting elements, each of them being equal  to $t^{-1}q_{ij}^rt$ for some $r$.  This contradicts to the form of the elements $P_i$.

In the first case,  since quasigeodesics $P_{i_j}$ and $Q_{i_j}$  are $2D$ close to each other, two  pieces in $P_{i_j}$ of length greater than $\epsilon |P_{i_j}|$ are $4D$ close to each other. Then, similarly to the previous case, we can show that there exists an element $t$ conjugating a power of $a$ to a power of $a$ and a power of $b$ to a power of $b$, and, therefore, commuting with both $a$ and $b$. Therefore, $t=1$ and two pieces in $P_{ij}$ of length greater than $\epsilon _1|P_{i_j}|$ for some $0\leq \epsilon _1\leq \epsilon$ coincide. This again contradicts to the form of the elements $P_{i_j}$. The theorem is proved.
\end{proof}

We complete this section with the following lemma which shows that a finite system of equations in $G$ is equivalent to one equation.
\begin{lemma} \label{onehyp} Let $G$ be a  torsion-free $\delta$-hyperbolic group. Then

1) There exists a constant $n=n(G,\delta )$ such that  equation $x^ny^nz^n=1$ implies $[x,y]=[x,z]=[y,z]=1.$

2) A finite system of equations in $G$ is equivalent to one equation.
\end{lemma}

\begin{proof} The first statement  follows from \cite{G}, 5.3B and \cite{Arzh}, Corollary 6. Indeed, one has to take $N$ to be the maximum of $100\delta$  and the numbers determined as in \cite{Arzh}, Corollary 6 for all triples of elements in $G$ of length not more than $100\delta$.
It can be also obtained from the proof of Theorem 1.4,  \cite{chai}. The theorem, in particular, states that in a non-elementary hyperbolic group, for any finite set of elements $x_1,\ldots ,x_k$ there
exists an integer $N$ such that the normal closure of $x_1^N,\ldots ,x_k^N$ is free. The number $N$ in the proof depends only on the number of elements $k$ and the minimum of their translation lengths. Since hyperbolic groups are translationally discrete, and $k=3$, this number depends only on the group and $\delta$. 

To prove the second statement we fix $n>N$, and elements $a,b\in G$ that do not commute, and show that the equation $(x^na)^na^{-n}=((yb)^nb^{-n})^n$ has only the trivial solution $x=1$ and $y=1$ in $G$.

Suppose, by contradiction, that $x\neq 1$. By the first statement of the lemma, $[x^na,a]=1$. Hence $[x^n,a]=1$. By transitivity of commutation (here we use $x\neq 1$) we have $[x,a]=1$. Therefore, we can rewrite the equation in the form $(x^n)^n=((yb)^nb^{-n})^n,$ which implies that
$[x^n,((yb)^nb^{-n})]=1$, and hence (since $G$ is torsion free), that $x^n=(yb)^nb^{-n}.$ Again, it follows that $[y,b]=1$, $[x,y]=1$ and $x^n=y^n.$ This implies that $x=y$ and, by transitivity of commutation, $[a,b]=1$, which contradicts the choice of $a,b.$ The contradiction shows that $x=1$. In this event the initial equation transforms into $(y^nb)^nb^{-n}=1$, which implies $[y,b]=1$ and $y=1$, as desired.

\end{proof}
\section{ Appendix: groups with no sufficient splitting and conjunctive $\exists\forall$-formulas}
To make this paper more self-contained we will give more explanations  how formulas (\ref{39}) and (\ref{40}) are obtained in \cite{KMel}.
In Section 5.4 of \cite{KMel} we defined the notion of a sufficient splitting of a group $K$ modulo a class of subgroups ${\mathcal K}$.
Let $F$ be a free group with basis $A$, $P=A\cup \{p\}$, $H=<p>\ast F.$ Let ${\mathcal K}$ consist of one subgroup ${\mathcal K}=\{H\}.$
The set of specializations $p$ such that  formula (\ref{38})
 is true in $F$ is associated with  a finite number of groups without sufficient splitting modulo $H$ and for each such group $K$ with a given combination of Max-classes of algebraic solutions. The total number of such classes is bounded by Theorem 11 in \cite{KMel}.
 Let $H\leq K$ and $K=<X, P|S(X,P)>$ does not have a sufficient splitting modulo $H$. Let $D$ be an abelian JSJ decomposition of $K$ modulo $H$.

We recall the notion of algebraic solutions. Let $K_1$  be a fully residually free quotient of the group $K$, $\kappa: K \rightarrow
K_1$ the canonical epimorphism, and $H_1 = H^\kappa$ the canonical
image of $H$ in $K_1$.
An elementary abelian splitting of $K_1$ modulo $H_1$ which does not
lift into  $K$ is called a {\em new} splitting.
\begin{definition}
\label{de:reducing} (Definition 20 \cite{KMel}) In the notation above the quotient $K_1$ is
called {\em reducing} if one of the following holds:
\begin{enumerate}
\item $K_1$ has  a non-trivial free decomposition modulo $H_1$;
\item $K_1$ has a new elementary abelian splitting modulo $H_1$.

\end{enumerate}
\end{definition}

 We
say that a homomorphism $\phi :K\rightarrow K_1$ is {\em special}
if $\phi $ either maps an edge group of $D$ to the identity  or
maps
 a non-abelian vertex group of $D$ to an abelian subgroup.

Let ${\mathcal R}=\{K/R(r_1), \ldots, K/R(r_s)\}$ be a complete
reducing system for $K$ (see \cite{KMel}).
Now we define algebraic and reducing solutions of $S = 1$ in $F$
with respect to $\mathcal{R}$. Let $\phi :H \rightarrow F$ be a
fixed $F$-homomorphism and $Sol_\phi$ the set of all homomorphisms
from $K$ onto $F$ which extend $\phi$.
 A solution $\psi \in Sol_\phi$ is called {\em
reducing}
  if there exists a solution  $\psi^\prime \in Sol_\phi$  in the $\sim_{MAX}$-equivalence class
  of $\psi$ which  satisfies one of the equations $r_1= 1, \ldots, r_k = 1$ (more precise, which factors through the corrective extension of one of the groups $K/R(r_i)$).
All non-reducing non-special solutions from $Sol_\phi$ are called
{\em $K$-algebraic} (modulo $H$ and $\phi$).

\begin{theorem}Let $H\leq K$ be as above. The fact that for parameters $P$ there are exactly $N$ non-equivalent Max-classes of $K$-algebraic solutions
of the equation $S(X,P)=1$ modulo $H$ can be written as a boolean combination of conjunctive $\exists\forall$-formulas.
\end{theorem}
The generating set
$X$ of $K$ corresponding to the decomposition $D$ can be partition
 as $X=X_1\cup X_2$ such that $G=<X_2\cup P>$ is the  fundamental group of the graph of groups obtained from $D$ by removing all QH-subgroups.
If $c_e$ is a given generator of an edge group of $D$, then we know how AE-transformation $\sigma$ associated with edge $e$ acts on the generators from the set $X$. Namely, if $x\in X$ is a generator of a vertex group, then either $x^{\sigma}=x$ or $x^{\sigma}=c^{-m}xc^{m},$ where $c$ is a root of the image of $c_e$ in $F$, or in case $e$ is an edge between abelian and rigid vertex groups and $x$ belongs to the abelian vertex group, $x^{\sigma}=xc^{m}$. Similarly, if $x$ is a stable letter then either $x^{\sigma}=x$ or $x^{\sigma}=xc^{m}.$

 One can write elements $c_e$ as words in generators $X_2$, $c_e=c_e(X_2).$ Denote $T=\{t_i,\ i=1,\ldots ,m\}.$ Consider a formula
\begin{multline*}
\exists X_1 \exists X_2\forall Y \forall T\forall Z \left(
S(X_1,X_2,P)=1\right.\\  \wedge \neg \left(\left.\bigwedge_{i=
1}^{m}[t_i,c_i(X_2)]=1 \wedge Z=X_2^{\sigma_T}\wedge
S(Y,X_2,P)=1\wedge V(Y,Z,P)=1\right)\right).
\end{multline*}
It says that there exists a solution of the equation $S(X_1,X_2,P)=1$ that is not Max-equivalent to a solution $Y,Z,P$ that
satisfies $V(Y,Z,P)=1$.  If now $V(Y,Z,P)=1$ is a disjunction of equations defining (corrective extensions of) maximal reducing quotients, then this formula states that for
parameters $P$ there exists at least one Max-class of algebraic solutions of $S(X,P)=1$ with respect to $H$.

Denote $\tau(T,X_2,Y,Z)=\left(\bigwedge_{i=
1}^{m}[t_i,c_i(X_2)]=1 \wedge Z=X_2^{\sigma_T}\wedge
S(Y,X_2,P)=1\wedge V(Y,Z,P)=1\right)$. The following formula states  that for
parameters $P$ there exists at least two non-equivalent Max-classes of algebraic solutions of $S(X,P)=1$ with respect to $H$.
\begin{multline*}
\theta _2(P)=\exists X_1, X_3 \exists X_2, X_4\forall Y, Y' \forall T, T',T''\forall Z,Z' \left(
S(X_1,X_2,P)=1\wedge S(X_3,X_4,P)=1\right.\\  \wedge \neg \left(\left.\tau(T,X_2,Y,Z)\vee \tau(T',X_4,Y',Z')\vee (\bigwedge_{i
=1}^{m}[{t_i}'',c_i(X_2)]=1\wedge X_2^{\sigma_{T''}}=X_4) \right)\right).
\end{multline*}

Both these formulas have type (\ref{39}). Similarly one can write a formula $\theta _N(P)$ that states for
parameters $P$ there exists at least N non-equivalent Max-classes of algebraic solutions of $S(X,P)=1$ with respect to $H$.

Then $\theta _N(P)\wedge \neg\theta _{N+1}(P)$ states that there are exactly $N$ non-equivalent Max-classes. The theorem is proved.

We will recall now how formulas (\ref{39}),(\ref{40}) appear in \cite{KMel}. In Section 12 in \cite{KMel}, parameters $P$ always correspond to variables $X_1,Y_1,\ldots ,X_{k-1}$. The tree $T_{X_k}$ is finite, we can have schemes of levels $(1,0),(1,1,),(2,1),(2,2)$ etc up to some number $(m,m)$.
We will concentrate on level $(2,1).$ In Definition 27 we define initial fundamental sequences of levels $(2,1)$ and $(2,2)$ and width $i$ (the possible width is bounded) modulo $P=<X_1,Y_1,\ldots ,X_{k-1}>$. All conditions (1)-(5) in this definition can be expressed by boolean combinations of conjunctive $\exists\forall$ formulas (depending on $X_1,Y_1,\ldots , X_{k-1}$). Lemmas 27 and 28  reduce the analysis of the set of parameters $P=<X_1,Y_1,\ldots ,X_{k-1}>$ for which there exists a fundamental sequence of level $(2,1)$ and width $i$ to those for which for this fundamental sequence there exists a generic family of
values of the variables  $Y_{k-1}^{(1)},
Z_{k-1}^{(2,j,s)},Y_{k-1}^{(2,j,s)}, \newline j=1,\ldots ,i_s, \
s=1,\ldots ,t,\ Z_{k-1}^{(3,k)}$ with properties (1)--(5) from
Definition 27. Similarly, Lemma 29 reduces the analysis of the set of parameters $P=<X_1,Y_1,\ldots ,X_{k-1}>$ for which there exists a fundamental sequence of level $(2,2)$ and width $i$ to those for which for this fundamental sequence there exists a generic family of
values of the variables  $Y_{k-1}^{(1)},
Z_{k-1}^{(2,j,s)},Y_{k-1}^{(2,j,s)}, \newline j=1,\ldots ,i_s, \
s=1,\ldots ,t,\ Z_{k-1}^{(3,k)}$ with properties (1)--(6) from
Definition 28.

For a given value of $X_1,Y_1,\ldots ,X_{k-1}$ formula
\begin{equation} \label{333}
 \Psi=\forall
Y_{k-1}\exists X_{k}\forall Y_k (U(X_1,Y_1,\ldots
,X_k,Y_k)=1\rightarrow V(X_1,Y_1,\ldots ,X_k,Y_k)=1)
\end{equation}
cannot be proved on level less than (2,1)
 if and only if  the following
conditions are satisfied.
\begin{enumerate}

\item [(a)] There exist algebraic solutions  for some
$U_{i,coeff}=1$ corresponding to the terminal group of a fundamental
sequence $V_{i,\rm fund}$ satisfying possibility (i) in Lemma
27, namely for this fundamental sequence there exists a generic family of
values of the variables  $Y_{k-1}^{(1)},
Z_{k-1}^{(2,j,s)},Y_{k-1}^{(2,j,s)}, \newline j=1,\ldots ,i_s, \
s=1,\ldots ,t,\ Z_{k-1}^{(3,k)}$ with properties (1)--(5) from
Definition 27.

\item [(b)] There is no algebraic solutions for equations
corresponding to the terminal groups of fundamental sequences that
describe solutions from $V_{i,\rm fund}$ that do not satisfy one of
properties (1)-(5). There is a finite number of such fundamental
sequences.

\item  [(c)] There is no algebraic solutions for equations
corresponding to the terminal groups of fundamental sequences of
level (2,1) and greater depth derived from $V_{i,\rm fund}$.

\item  [(d)] $(X_1,Y_1,\ldots ,Y_{k-1},Y_{k-1}^{(1)})$ cannot be
extended to a solution of  $V=1$ by arbitrary $X_k$ ($X_k$ of level
0) and $Y_k$ of level (1,0).
\end{enumerate}

For a given value of $X_1,Y_1,\ldots ,X_{k-1}$ formula (\ref{333})
can be disproved on level $(2,1)$ if and only if it cannot be proved
on level less than $(2,1)$ and there is no algebraic solutions for
equations corresponding to the terminal groups of fundamental
sequences of level (2,2) corresponding to $V_{i,\rm fund}$.
These conditions can be described by a boolean combination of conjunctive $\exists\forall$-formulas of type (\ref{39}).

\section*{Acknowledgments}

The authors would like to thank the referee for many comments and suggestions. The first named author gratefully acknowledges  support from the NSF grant DMS-1201379 and PSC-CUNY enhanced award. The second named author gratefully acknowledges  support from the  NSF grant DMS-1201379.

\end{document}